\documentclass[10pt]{amsart}
\usepackage{amsmath, amssymb, amsthm, graphicx}
\theoremstyle{plain}
\newtheorem{prop}{Proposition}
\newtheorem{lemm}[prop]{Lemma}

\theoremstyle{definition}
\newtheorem{defi}[prop]{Definition}
\newtheorem{rema}[prop]{Remark}

\renewcommand\aa{a}
\newcommand\bb{b}
\newcommand\BR[1]{B_{#1}}
\newcommand\cc{c}
\newcommand\CC{C}
\newcommand\Diag[1]{\mathcal{D}(#1)}
\let\ge=\geqslant
\newcommand\ie{{\it i.e.}}
\newcommand\ii{i}
\newcommand\inv{^{-1}}
\newcommand\jj{j}

\newcommand\mm{m}
\newcommand\Name[3]{\{#1,#2\}_{#3}}
\newcommand\nn{n}
\newcommand\pp{p}
\newcommand\qq{q}
\newcommand\resp{{\it resp}}
\newcommand\rr{r}
\newcommand\Seq[1]{S(#1)}
\newcommand\sig[1]{\sigma_{#1}^{\relax}}
\newcommand\sigg[2]{\sigma_{#1}^{#2}}
\newcommand\siginv[1]{\sigma_{#1}^{-1}}
\newcommand\sName[3]{\scriptstyle\{\!#1,#2\!\}_{\!#3}}
\renewcommand{\ss}{s}
\newcommand\ww{w}
\begin{document}

\hfill{\tiny 2009-06}

\author{Patrick DEHORNOY}
\address{Laboratoire de Math\'ematiques Nicolas Oresme,
Universit\'e de Caen, 14032 Caen, France}
\email{dehornoy@math.unicaen.fr}
\urladdr{//www.math.unicaen.fr/\!\hbox{$\sim$}dehornoy}

\title{Combinatorial distance between braid words}



\begin{abstract}
We give a simple naming argument for establishing lower
bounds on the combinatorial distance between (positive) braid words.
\end{abstract}

\maketitle

\footnote{Work partially supported by the ANR grant ANR-08-BLAN-0269-02}
It is well-known that, for $\nn \ge 3$, Artin's braid
group~$\BR\nn$, which is the group defined by the presentation
\begin{equation*}
\bigg\langle \sig1, ..., \sig{n-1} \ \bigg\vert\ 
\begin{matrix}
\sig\ii \sig j = \sig j \sig\ii 
&\text{for} &\vert i-j \vert\ge 2\\
\sig\ii \sig j \sig\ii = \sig j \sig\ii \sig j 
&\text{for} &\vert i-j \vert = 1
\end{matrix}
\ \bigg\rangle
\end{equation*}
has a quadratic Dehn function, \ie, there exist constants~$\CC_\nn,
\CC'_\nn$ such that, if $\ww$ is an $\nn$-strand braid word of 
length~$\ell$ that represents the unit braid, then the number of braid
relations needed to transform~$\ww$ into the empty word is at
most~$\CC_\nn
\ell^2$ and, on the other hand, there exists for each~$\ell$ at least one
length~$\ell$ word~$\ww$ such that the minimal
number of such braid relations is at least~$\CC'_\nn
\ell^2$---see for instance~\cite{Eps}.

In a recent posting~\cite{HKN}, Hass, Kalka, and Nowik developed a knot
theoretical argument for establishing lower bounds on the combinatorial
distance between two equivalent positive braid words, \ie, on the minimal
number of braid relations needed to transform the former into the latter.
Using some knot invariants introduced in~\cite{HaN}, they prove

\begin{prop}
\label{P:Main}
For each~$\mm$, the combinatorial distance between the (equivalent)
braid words~$\sigg1{2\mm} (\sig2\sigg12\sig2)^\mm$
and~$(\sig2\sigg12\sig2)^\mm\sigg1{2\mm}$ equals~$4\mm^2$.
\end{prop}

The purpose of this note is to observe that the above result also follows
from the direct combinatorial argument similar to the one developed
in~\cite{Dhx} for the reduced expressions of a permutation.

By definition, a braid word~$\ww$ is a finite sequence of letters~$\sig\ii$
and their inverses, and it naturally encodes a braid diagram~$\Diag\ww$
once one decides that $\sig\ii$ encodes the elementary diagram in
which the $(\ii+1)$st strand passes over the $\ii$th strand. For instance,
the diagrams associated with the words of Proposition~\ref{P:Main} are
displayed in Figure~\ref{F:Diagrams}.

For simplicity, we restrict to positive braid words, \ie, words that contain
no letter~$\siginv\ii$---see Remark~\ref{R:Negative} below. Each strand in
a braid diagram has a well-defined initial position, hereafter called its
\emph{name}, and we can associate with each crossing of the diagram,
hence with each letter in the braid word that encodes it,  the names of the
strands involved in the crossing. As two strands may  cross more than once,
we shall also include the rank of the crossing, thus using the name
$\Name\pp\qq\aa$ for the $\aa$th crossing of the strands with initial
positions~$\pp$ and~$\qq$. In this way, we associate with each
positive braid word a sequence of names:

\begin{defi}
(See Figure~\ref{F:Diagrams}.)
For~$\ww$ a positive braid word, the sequence~$\Seq\ww$ is defined to be
empty if $\ww$ is the empty word and, for $\ww = \ww'
\sig\ii$, to be the sequence obtained from~$\Seq{\ww'}$ by appending
$\Name\pp\qq\aa$, where $\pp$ and~$\qq$ are the initial positions of
the strands that finish at position~$\ii$ and~$\ii+1$ in~$\Diag{\ww'}$
and $\aa-1$ is the number of times the latter strands cross
in~$\Diag{\ww'}$. 
\end{defi}

\begin{figure}[htb]
\begin{picture}(96,44)(0,0)
\put(0,0){\includegraphics{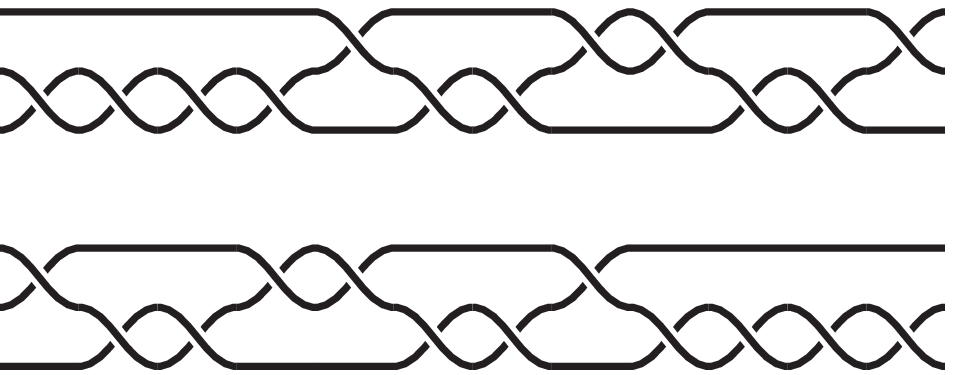}}
\put(-3,0){$1$}
\put(-3,6){$2$}
\put(-3,12){$3$}
\put(-3,24){$1$}
\put(-3,30){$2$}
\put(-3,36){$3$}
\put(0,16){$\sName231$}
\put(8,16){$\sName131$}
\put(16,16){$\sName132$}
\put(24,16){$\sName232$}
\put(32,16){$\sName233$}
\put(40,16){$\sName133$}
\put(48,16){$\sName134$}
\put(56,16){$\sName234$}
\put(64,16){$\sName121$}
\put(72,16){$\sName122$}
\put(80,16){$\sName123$}
\put(88,16){$\sName124$}
\put(0,40){$\sName121$}
\put(8,40){$\sName122$}
\put(16,40){$\sName123$}
\put(24,40){$\sName124$}
\put(32,40){$\sName231$}
\put(40,40){$\sName131$}
\put(48,40){$\sName132$}
\put(56,40){$\sName232$}
\put(64,40){$\sName233$}
\put(72,40){$\sName133$}
\put(80,40){$\sName134$}
\put(88,40){$\sName234$}
\end{picture}
\caption{\sf Braid diagrams associated with the two braid words of
Proposition~\ref{P:Main} (here with $\mm = 2$), together with the
associated sequences of names.}
\label{F:Diagrams}
\end{figure}

The fact that three (\resp. four) different strands are involved in a braid
$\sig\ii \sig{\ii+1} \sig\ii$ (\resp. $\sig\ii \sig\jj$ with $\vert\ii-\jj\vert
\ge 2$) and the explicit definition of the names immediately imply

\begin{lemm}
\label{L:Permut}
Assume that $\ww, \ww'$ are $\nn$-strand braid words and $\ww'$ is
obtained from~$\ww$ by applying one braid relation $\sig\ii \sig{\ii+1}
\sig\ii = \sig{\ii+1} \sig\ii \sig{\ii+1}$ (\resp. $\sig\ii \sig\jj = \sig\jj
\sig\ii$ with $\vert \ii - \jj\vert \ge 2$). Then there exist pairwise distinct
numbers~$\pp, \qq, \rr$ in~$\{1, ..., \nn\}$ and integers~$\aa, \bb, \cc$
(\resp. pairwise distinct $\pp, \qq, \rr, \qq$ and integers~$\aa, \bb$) such
that $\Seq{\ww'}$ is obtained from~$\Seq\ww$ by reversing some
subsequence $(\Name\pp\qq\aa, \Name\pp\rr\bb, \Name\qq\rr\cc)$
(\resp. reversing some subsequence $(\Name\pp\qq\aa,
\Name\rr\ss\bb)$). 
\end{lemm}

Then Proposition~\ref{P:Main} is easy.

\begin{proof}[Proof of Proposition~\ref{P:Main}]
Figure~\ref{F:Kampen} below makes the upper bound trivial, so we only
have to prove a lower bound result. Let $\ww_\mm$ and $\ww'_\mm$ be
the involved braid words. We consider the entries of the
form~$\Name12\aa$ and~$\Name23\bb$ in~$\Seq{\ww_\mm}$
and~$\Seq{\ww'_\mm}$. In~$\Seq{\ww_\mm}$, they appear in the
order $\Name121, ..., \Name12{2\mm}, \Name231, ...,
\Name23{2\mm}$, whereas in~$\Seq{\ww'_\mm}$ they appear in the
order $\Name231, ..., \Name23{2\mm}, \Name121, ..., \Name12{2\mm}$.
By Lemma~\ref{L:Permut}, applying one braid relation can only switch two
entries in these sequences. Therefore, at least $4\mm^2$ braid relations
(of the form $\sig1 \sig2 \sig 1 = \sig2 \sig1 \sig2$) are needed to
exchange the $2\mm$~entries~$\Name12\aa$ and the
$2\mm$~entries~$\Name23\bb$.
\end{proof} 

We conclude with a few additional observations.

\begin{rema}
Each derivation from a (positive) braid word to another equivalent one can
be illustrated using a van Kampen diagram, which is a planar diagram
tesselated by tiles corresponding to braid relations---see for
instance~\cite{Eps}, and Figure~\ref{F:Kampen} below. For each
name~$\Name\pp\qq\aa$ occurring in the diagram, connecting all edges
having that name yields a family of transversal curves called
\emph{separatrices} in~\cite{Dhx}. It is easy to check that, if any two
separatrices of a van Kampen diagram cross at most once, then the diagram
must be optimal, in the sense that the number of faces is the minimal
possible one,
\ie, it achieves the combinatorial distance. This criterion is clearly satisfied
in the case of Figure~\ref{F:Kampen}. As the diagram has $4\mm^2$
faces, we conclude that the combinatorial distance between the bounding
words is exactly~$4\mm^2$.
\end{rema}

\begin{figure}[htb]
\begin{picture}(66,72)(0,0)
\put(0,0){\includegraphics{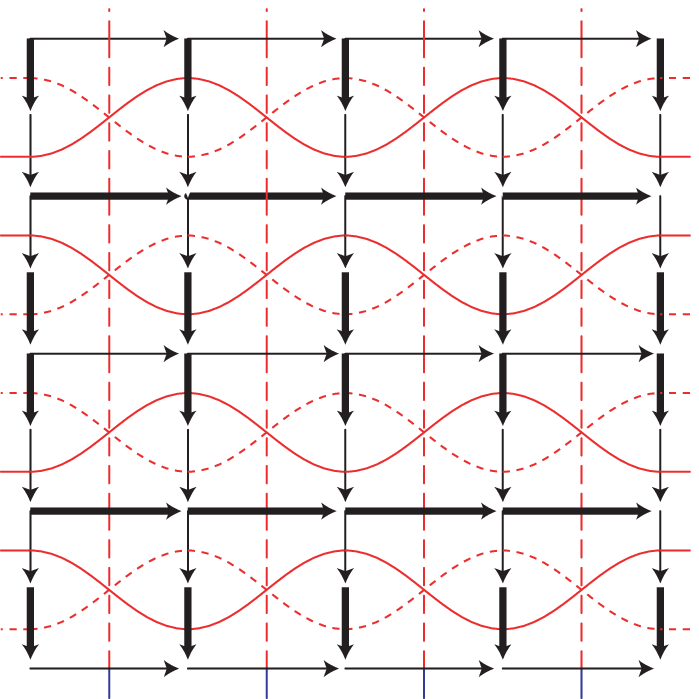}}
\put(8,71){$\sName121$}
\put(24,71){$\sName122$}
\put(40,71){$\sName123$}
\put(56,71){$\sName124$}
\put(-7,63){$\sName231$}
\put(-7,55){$\sName131$}
\put(-7,47){$\sName132$}
\put(-7,39){$\sName232$}
\put(-7,31){$\sName233$}
\put(-7,23){$\sName133$}
\put(-7,15){$\sName134$}
\put(-7,7){$\sName234$}
\end{picture}
\caption{\sf A van Kampen diagram witnessing that the two braid words of
Proposition~\ref{P:Main} are equivalent (here with $\mm = 2$). Thin edges
represent~$\sig1$, thick edges represent~$\sig2$. The
dotted (red) lines connect the edges bearing the same name;
as any two of them cross at most once, the diagram  achieves the
combinatorial distance.}
\label{F:Kampen}
\end{figure}

\begin{rema}
\label{R:Negative}
The current approach can be easily extended to arbitrary braid words. The
name attributed to a negative crossing~$\siginv\ii$ has to be defined to be
$\Name\pp\qq\aa\inv$, where $\pp, \qq$ still are the initial positions of
the strands that cross, and $\aa$ is the number of earlier crossings of
these strands, counted algebraically, \ie, it is the linking number of these
two strands so far. For instance, the sequence $\Seq{\sig1 \sigg2{-3} \sig2
\sig1}$ is $(\Name121, \Name130\inv, \Name13{-1}\inv,
\Name13{-2}\inv, \Name13{-2}, \Name122)$. Then Lemma~\ref{L:Permut}
remains valid, as the contribution of the free group relations $\sig\ii
\siginv\ii = \siginv\ii \sig\ii = 1$ consists in creating or deleting a
subsequence of the form $(\Name\pp\qq\aa, \Name\pp\qq\aa\inv)$ or
$(\Name\pp\qq\aa\inv, \Name\pp\qq\aa)$.
\end{rema}

\begin{rema}
A more symmetric and still more obvious example implying a number of
relations that is quadratic with respect to the length of the initial words is
obtained by starting with~$\sigg1{2\mm}$ and~$\sigg2{2\mm}$ and
completing them into their least common right multiple. Then the
sequence associated with the first braid word must begin with $\Name121,
..., \Name12{2\mm}$, whereas that associated with the second one must
begin with $\Name231, ..., \Name23{2\mm}$, so,
by Lemma~\ref{L:Permut} again, $4\mm^2$ braid relations are certainly
needed to transform one into the other---see Figure~\ref{F:KampenBis} for
an illustration in terms of van Kampen diagram and separatrices.
\end{rema}

\begin{figure}[htb]
\begin{picture}(72,72)(0,0)
\put(0,0){\includegraphics{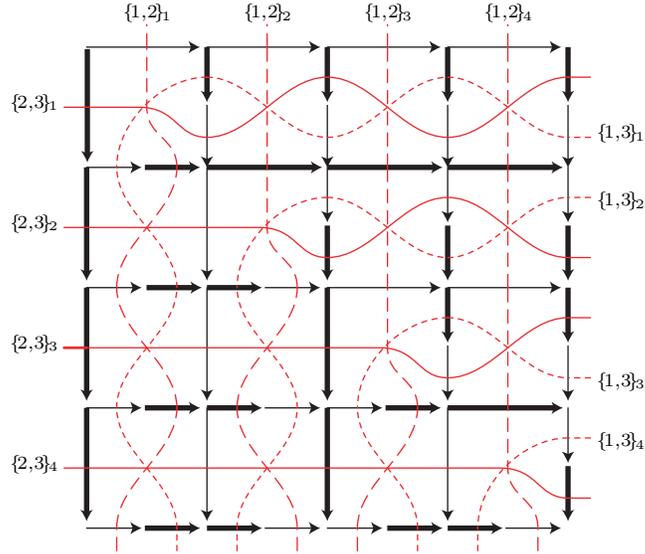}}
\put(8,71){$\sName121$}
\put(24,71){$\sName122$}
\put(40,71){$\sName123$}
\put(56,71){$\sName124$}
\put(-7,59){$\sName231$}
\put(-7,43){$\sName232$}
\put(-7,27){$\sName233$}
\put(-7,11){$\sName234$}
\put(71,55){$\sName131$}
\put(71,46){$\sName132$}
\put(71,22){$\sName133$}
\put(71,14){$\sName134$}
\end{picture}
\caption{\sf A  van Kampen diagram witnessing that the braid
words $\sigg1{2\mm}(\sig2\sigg12 \sig2)^\mm$ and  
$\sigg2{2\mm}(\sig1\sigg22 \sig1)^\mm$ are equivalent and lie at
combinatorial distance $4\mm^2$ (here with
$\mm = 2$). As in Figure~\ref{F:Kampen}, the dotted (red) lines connect
the edges bearing the same name. Any two of them cross at most once,
hence the diagram is optimal.}
\label{F:KampenBis}
\end{figure}

\end{document}